\documentclass[letterpaper]{amsart}

\newtheorem{theorem}{Theorem}[section]

\newtheorem{proposition}[theorem]{Proposition}

\theoremstyle{definition}

\theoremstyle{remark}
\newtheorem{remark}[theorem]{Remark}

\usepackage[pdftex]{graphicx}
\usepackage{color}
\usepackage{amsmath}
\usepackage{amsfonts}
\usepackage{amssymb}
\usepackage{epsfig}
\usepackage{rotating}
\usepackage{graphics}
\newcommand{\be}{\begin{equation}}

\newcommand{\ee}{\end{equation}}

\newcommand{\tde}{{\stackrel{\scriptscriptstyle{DE}}{T}}\phantom{}}

\newcommand{\dz}{\wedge}

\newcommand{\ba}{\begin{array}}

\newcommand{\ea}{\end{array}}

\newcommand{\beq}{\begin{eqnarray}}

\newcommand{\eeq}{\end{eqnarray}}

\newtheorem{lm}{lemma}

\newtheorem{thee}{theorem}

\newtheorem{proo}{proposition}

\newtheorem{co}{corollary}

\newtheorem{rem}{remark}

\newtheorem{deff}{definition}

\newcommand{\bd}{\begin{deff}}

\newcommand{\ed}{\end{deff}}

\newcommand{\bl}{\begin{lm}}

\newcommand{\el}{\end{lm}}

\newcommand{\bp}{\begin{proo}}

\newcommand{\ep}{\end{proo}}

\newcommand{\bt}{\begin{thee}}

\newcommand{\et}{\end{thee}}

\newcommand{\bc}{\begin{co}}

\newcommand{\ec}{\end{co}}

\newcommand{\brm}{\begin{rem}}

\newcommand{\erm}{\end{rem}}

\newcommand{\der}{{\rm d}}

\usepackage{t1enc}
\def\frak{\mathfrak}

\newcommand{\newc}{\newcommand}

\let\ccdot\cdot
\def\cdot{\hbox to 2.5pt{\hss$\ccdot$\hss}}

\newc{\aR}{\mbox{\boldmath{$ R$}}}
\newc{\aS}{\mbox{\boldmath{$ S$}}}
\newc{\aT}{\mbox{\boldmath{$ T$}}}
\newc{\aW}{\mbox{\boldmath{$ W$}}}

\newc{\aK}{\mbox{\boldmath{$ K$}}}
\newc{\aL}{\mbox{\boldmath{$ L$}}}
\usepackage{amssymb}
\usepackage{amscd}

\newcommand{\hook}{\raisebox{-0.35ex}{\makebox[0.6em][r]
{\scriptsize $-$}}\hspace{-0.15em}\raisebox{0.25ex}{\makebox[0.4em][l]{\tiny
 $|$}}}

\newcommand{\bma}{\begin{pmatrix}}
\newcommand{\ema}{\end{pmatrix}}

\newc{\obstrn}[2]{B^{#1}_{#2}}

\newcommand{\rpl}                         % +) or <+
{\mbox{$
\begin{picture}(12.7,8)(-.5,-1)
\put(0,0.2){$+$}
\put(4.2,2.8){\oval(8,8)[r]}
\end{picture}$}}

\newcommand{\lpl}                         % (+ or +>
{\mbox{$
\begin{picture}(12.7,8)(-.5,-1)
\put(2,0.2){$+$}
\put(6.2,2.8){\oval(8,8)[l]}
\end{picture}$}}

\usepackage{ifthen}

\newc{\tensor}[1]{#1}
\newc{\Mvariable}[1]{\mbox{#1}}
\newc{\down}[1]{{}_{#1}}
\newc{\up}[1]{{}^{#1}}

\newc{\JulyStrut}{\rule{0mm}{6mm}}
\newc{\midtenPan}{\mbox{\sf S}}
\newc{\midten}{\mbox{\sf T}}
\newc{\midtenEi}{\mbox{\sf U}}
\newc{\ATen}{\mbox{\sf E}}
\newc{\BTen}{\mbox{\sf F}}
\newc{\CTen}{\mbox{\sf G}}

\def\sideremark#1{\ifvmode\leavevmode\fi\vadjust{\vbox to0pt{\vss% the remark
 \hbox to 0pt{\hskip\hsize\hskip1em%                          will appear only
 \vbox{\hsize3cm\tiny\raggedright\pretolerance10000%          on the side
 \noindent #1\hfill}\hss}\vbox to8pt{\vfil}\vss}}}%
\newcounter{romenumi}
\newcommand{\labelromenumi}{(\roman{romenumi})}

\begin{document}
\title{Is dark energy meaningless?}
\vskip 1.truecm
\author{Pawe\l~ Nurowski} \address{Instytut Fizyki Teoretycznej,
Uniwersytet Warszawski, ul. Hoza 69, Warszawa, Poland}
\email{nurowski@fuw.edu.pl} %\thanks{This research was supported by
                            %the Polish grant}

\date{\today}

\begin{abstract}
We show that there are isometrically nonequivalent Robertson-Walker
metrics which have the same set of geodesics. While one of these
metrics satisfies the Einstein equations of pure dust without a 
cosmoological constant, all the other describe pure dust with
additional energy momentum tensor of cosmological constant type. Since
each of these metrics have the same geodesics it is not clear how to
distinguish experimentally between the Universes whose energy momentum
tensor includes or not the cosmological constant type term.

%\vskip5pt\centerline{\small\textbf{MSC classification}: ...}\vskip15pt
%\vskip5pt\centerline{\small\textbf{Key words}: ...}\vskip15pt
\end{abstract}
\maketitle
%*************
%\tableofcontents
\newcommand{\bbS}{\mathbb{S}}
\newcommand{\bbR}{\mathbb{R}}
\newcommand{\sog}{\mathbf{SO}}
\newcommand{\slg}{\mathbf{SL}}
\newcommand{\glg}{\mathbf{GL}}
\newcommand{\og}{\mathbf{O}}
\newcommand{\soa}{\frak{so}}
\newcommand{\sla}{\frak{sl}}
\newcommand{\sua}{\frak{su}}
\newcommand{\dr}{\mathrm{d}}
\newcommand{\sug}{\mathbf{SU}}
\newcommand{\gat}{\tilde{\gamma}}
\newcommand{\Gat}{\tilde{\Gamma}}
\newcommand{\thet}{\tilde{\theta}}
\newcommand{\Thet}{\tilde{T}}
\newcommand{\rt}{\tilde{r}}
\newcommand{\st}{\sqrt{3}}
\newcommand{\kat}{\tilde{\kappa}}
\newcommand{\kz}{{K^{{~}^{\hskip-3.1mm\circ}}}}
\newcommand{\bv}{{\bf v}}
\newcommand{\di}{{\rm div}}
\newcommand{\curl}{{\rm curl}}
\newcommand{\cs}{(M,{\rm T}^{1,0})}
\newcommand{\tn}{{\mathcal N}}
%*************
%\section{}
To interpret the cosmological data one has to assume a model of
space-time, which according to the current paradigm, is a
4-dimensional manifold $M$ equipped with the Robertson-Walker metric
$g$ given by 
\be
g=-\der t^2+R^2\frac{\der x^2+\der y^2+\der z^2}{1+\tfrac{\kappa}{4}(x^2+y^2+z^2)},\quad\quad\kappa=+1,0,-1.\label{r1}\ee
Here $R=R(t)$ is a real function (the scale factor), of the cosmic
time $t$. 
In the following we use an orthonormal coframe $\theta^\mu$,
$\mu=0,1,2,3$, for $g$. This is given by 
\be
\theta^0=\der t,\quad\theta^i=\frac{R\der
  x^i}{1+\tfrac{\kappa}{4}(x^2+y^2+z^2)},\quad
x^i=(x,y,z),\label{cof}\ee
and in it the metric $g$ reads:
$$g=g_{\mu\nu}\theta^\mu\theta^\nu=-{\theta^0}^2+{\theta^1}^2+{\theta^2}^2+{\theta^3}^2.$$

In this letter we observe that each Robertson-Walker spacetime
$(M,g)$, admits a 1-parameter family of metrics $\tilde{g}$, which are not
isometric to $g$, but which have the same set of geodesics as
$(M,g)$. Then we speculate about the consequences of using $\tilde{g}$
rather than $g$ to interpret the cosmological data. In particular, we
show that a pure dust without a cosmological constant in the
Robertson-Walker metric $g$, can be interpreted as a pure dust with
energy momentum tensor of cosmological constant type (dark energy), 
$\tde_{\mu\nu}=-\tfrac{1}{8\pi
  G}\tilde{\Lambda}\tilde{g}_{\mu\nu},$ in the corresponding metric $\tilde{g}$.

To see this we do as follows:

Consider a 1-parameter family of metrics $\tilde{g}$ on $M$ related to
$g$ in (\ref{r1}) by:
\be
\tilde{g}=-\frac{{\theta^0}^2}{(1-s R^2)^2}+ \frac{{\theta^1}^2}{1-s
  R^2}+ \frac{{\theta^2}^2}{1-s R^2}+ \frac{{\theta^3}^2}{1-s R^2},\label{r2}\ee
where $s$ is a real constant. Then we have the following theorem
\begin{theorem}
For each value of the real parameter $s$ the metric (\ref{r2}) has on
$M$ the same unparametrised geodesics as the Roberston-Walker metric (\ref{r1}). 
\end{theorem}     
\begin{proof}
It is well known \cite{bde,car,ea,ema,nn,nur} that two metrics $g$ and $\tilde{g}$ have the same
unparametrized geodesics if and only if their respective Levi-Civita connections 
$\nabla$ and $\tilde{\nabla}$ are related via:
$$\tilde{\nabla}_XY=\nabla_XY+A(X)Y+A(Y)X,\quad\quad\quad\forall X,Y\in{\rm T}M$$
with some 1-form $A$ on $M$.   

For our pourposes it is convenient to describe a Levi-Civita 
connection $\nabla$ of a metric $g=g_{\mu\nu}\theta^\mu\theta^\nu$ 
in terms of the connection 1-forms $\Gamma^\mu_{~\nu}$ associated to
the coframe $\theta^\mu$ via:
$$\der\theta^\mu+\Gamma^\mu_{~\nu}\dz\theta^\nu=0,\quad\quad \der
g_{\mu\nu}-\Gamma_{\mu\nu}-\Gamma_{\nu\mu}=0,\quad\quad
\Gamma_{\mu\nu}=g_{\mu\rho}\Gamma^\rho_{~\nu}.$$ In particular we have 
$\Gamma_{\mu\nu}=g(X_\mu,\nabla X_\nu)$, where $X_\mu$ is a frame dual
to $\theta^\nu$, $X_\mu\hook\theta^\nu=\delta^\nu_{~\mu}$.  

In terms of the connection 1-forms the two connections
$\tilde{\nabla}$ and $\nabla$ have the same unparametrised geodesics
iff there exists a coframe $\theta$ and a 1-form $A=A_\mu\theta^\mu$ on
$M$, such that
the corresponding connection 1-forms $\tilde{\Gamma}^\mu_{~\nu}$ and
$\Gamma^\mu_{~\nu}$ are related via:
\be
\tilde{\Gamma}^\mu_{~\nu}=\Gamma^\mu_{~\nu}+\delta^\mu_{~\nu}A+A_\nu\theta^\mu.\label{pt}\ee
in this coframe\footnote{The transformation of Levi-Civita connections
$\Gamma\to\tilde{\Gamma}$ is called a projective transformation. To
  see that two connections which are transformable to each other via
  projective transformations have the same geodesics is very easy: the
connection coeefficients $\tilde{\Gamma}^\mu_{~\nu\rho}$ defined by the
connection 1-forms via
$\tilde{\Gamma}^\mu_{~\nu}=\tilde{\Gamma}^\mu_{~\nu\rho}\theta^\rho$ define the
geodesic equation: $\frac{\der v^\mu}{\der t} +\tilde{\Gamma}^\mu_{~\nu\rho}v^\nu
v^\rho=0$. If we insert
$\tilde{\Gamma}^\mu_{~\nu\rho}=\Gamma^\mu_{~\nu\rho}+\delta^\mu_{~\nu}A_\rho+\delta^\mu_{~\rho}A_\nu$
in this equation we get $\frac{\der v^\mu}{\der t}+\Gamma^\mu_{~\nu\rho}v^\nu
v^\rho=-2 (v\cdot A) v^\mu$, i.e. again a geodesics equation, but now
for connection $\Gamma$ and in a different parametrization.}.

Thus to proof the theorem it is enough to find a common coframe and
$A$ such that the Levi-Civita connection 1-forms for metrics
(\ref{r1}) and (\ref{r2}) satisfy (\ref{pt}) with some $A$. 

It turns out that such a coframe is given by (\ref{cof}). Calculating
the Levi-Civita connection 1-forms $\Gamma^\mu_{~\nu}$ for $g$ as in
(\ref{r1}) in this coframe we find that
\be
\Gamma^\mu_{~\nu}=\bma
0&\frac{\dot{R}\theta^1}{R}&\frac{\dot{R}\theta^2}{R}&\frac{\dot{R}\theta^3}{R}\\&&&\\\frac{\dot{R}\theta^1}{R}&
0&\frac{-\kappa y\theta^1+\kappa x\theta^2}{2R}&\frac{-\kappa z\theta^1+\kappa x\theta^3}{2R}\\&&&\\\frac{\dot{R}\theta^2}{R}&\frac{\kappa y\theta^1-\kappa x\theta^2}{2R}&0&\frac{-\kappa z\theta^2+\kappa y\theta^3}{2R}\\&&&\\\frac{\dot{R}\theta^3}{R}&\frac{\kappa z\theta^1-\kappa x\theta^3}{2R}&\frac{\kappa z\theta^2-\kappa y\theta^3}{2R}&
0
\ema.
\ee
Calculations of $\tilde{\Gamma}^\mu_{~\nu}$ for (\ref{r2}) in this
coframe gives:
\be
\tilde{\Gamma}^\mu_{~\nu}=\bma
\frac{2s
  R\dot{R}\theta^0}{1-sR^2}&\frac{\dot{R}\theta^1}{R}&\frac{\dot{R}\theta^2}{R}&\frac{\dot{R}\theta^3}{R}\\&&&\\\frac{\dot{R}\theta^1}{R(1-sR^2)}&
\frac{sR\dot{R}\theta^0}{1-sR^2}&\frac{-\kappa y\theta^1+\kappa x\theta^2}{2R}&\frac{-\kappa z\theta^1+\kappa x\theta^3}{2R}\\&&&\\\frac{\dot{R}\theta^2}{R(1-sR^2)}&\frac{\kappa y\theta^1-\kappa x\theta^2}{2R}&
\frac{sR\dot{R}\theta^0}{1-sR^2}&\frac{-\kappa z\theta^2+\kappa y\theta^3}{2R}\\&&&\\\frac{\dot{R}\theta^3}{R(1-sR^2)}&\frac{\kappa z\theta^1-\kappa x\theta^3}{2R}&\frac{\kappa z\theta^2-\kappa y\theta^3}{2R}&
\frac{sR\dot{R}\theta^0}{1-sR^2}
\ema.
\ee
It is a matter of checking that the 1-form 
$$A=\frac{sR\dot{R}}{1-sR^2}\theta^0$$ 
is such that (\ref{pt}) holds for $\tilde{\Gamma}^\mu_{~\nu}$ and
$\Gamma^\mu_{~\nu}$. This finishes the proof.
\end{proof}
\begin{remark}
Note that if $s=0$ the metric $\tilde{g}$ coincides with $g$.
Observe also that the metrics $\tilde{g}$, belong to the
Robertson-Walker class for all values of $s$: one can bring them in the form
(\ref{r1}) by an apropriate redefinition of the coordinate $t$
and the function $R$. Thus associated with each Robertson-Walker
metric $g$ is an entire one parameter family of Robertson-Walker
metrics $\tilde{g}$, which includes $g$, and have the property that
all the metrics from this class have the same unparametrised geodesics
on $M$. The metrics $\tilde{g}$, as being Robertson-Walker metrics, are
all conformally flat. However for different values of $s$, such as
e.g. $s=0$ and $s=1$, they are
not isometric: their curvature, totally encoded in the Einstein
tensor, has different properties. 
\end{remark}
Calculation of the curvature $R^\mu_{~\nu\rho\sigma}$ and
$\tilde{R}^\mu_{~\nu\rho\sigma}$ and the Ricci tensors, 
$R_{\nu\sigma}=R^\mu_{~\nu\mu\sigma}$ and 
$\tilde{R}_{\nu\sigma}=\tilde{R}^\mu_{~\nu\mu\sigma}$, for the metrics $g$ and $\tilde{g}$,
still using the same coframe (\ref{cof}), yields the following
proposition.
\begin{proposition}\label{pro}
The respective 
Einstein tensors 
$E_{\mu\nu}=R_{\mu\nu}-\tfrac12Rg_{\mu\nu}$ and
$\tilde{E}_{\mu\nu}=\tilde{R}_{\mu\nu}-\tfrac12\tilde{R}\tilde{g}_{\mu\nu}$,
in coframe (\ref{cof}), read:
$$E_{\mu\nu}=\bma E_{00}&0\\
0&E_{ij}\ema
$$
with
$$
E_{00}=\frac{3(\kappa+\dot{R}^2)}{R^2},\quad\quad
E_{ij}=-\frac{\kappa+\dot{R}^2+2R\ddot{R}}{R^2}\delta_{ij},
$$
and
$$\tilde{E}_{\mu\nu}=\bma \tilde{E}_{00}&0\\
0&\tilde{E}_{ij}\ema,
$$
with
$$
\tilde{E}_{00}=\frac{1}{(1-s R^2)^2}\Big(E_{00}-3s\kappa\Big),\quad\quad
\tilde{E}_{ij}=\frac{1}{1-s R^2}\Big(E_{ij}+s(\kappa+2\dot{R}^2+2R\ddot{R})\delta_{ij}\Big).$$
\end{proposition}
Now we assume that the metric $g$ satisfies the Einstein equations 
\be
E_{\mu\nu}=8\pi G T_{\mu\nu},\label{ee}\ee
where $T_{\mu\nu}=\rho u_\mu u_\nu$ is the energy-momentum tensor of
pure dust with energy density $\rho$ and the 4-velocity $u=u^\mu
X_\mu$, orthogonal to the hypersurfaces $t=const$. This in particular 
means that in the frame $X_\mu$ dual to the coframe (\ref{cof}) we
have $$u^\mu=(1,0,0,0),$$ so that the Einstein equations (\ref{ee})
are:
\be\begin{aligned}
&E_{00}=\frac{3(\kappa+\dot{R}^2)}{R^2}=8\pi G\rho\\
&E_{ij}=-\frac{\kappa+\dot{R}^2+2R\ddot{R}}{R^2}\delta_{ij}=0.
\end{aligned}\label{fr1}\ee
Each solution to these equations satisfies the Friedmann equation
\be
\dot{R}^2=\frac{2GM}{R}-\kappa,\label{fr2}\ee
with a constant $M=\frac43\pi\rho R^3$. From now on we assume the
equations (\ref{fr1})-(\ref{fr2}) to be satisfied. 

Thus we have a Friedmann-Robertson-Walker Universe $(M,g)$ filled with
the comoving dust with 4-velocity $u$. 

Now if we forget about the parametrization of geodesics in this
Universe, and would like to reconstruct the metric from the analysys
of unparametrized geodesics we would equally use any metric
$\tilde{g}$ with whathever value of the parameter $s$. But if we
decided to use a metric $\tilde{g}$ with $s\neq 0$ we would noticed
that now our Universe satsifies quite a different Einstein equations
than these in (\ref{ee}). 

This is because of the folllowing line of arguments:

The vector field $u$ is not anymore a unit vector field in the metric
$\tilde{g}$. Actually $\tilde{g}(u,u)=-\frac{1}{(1-sR^2)^2}$. So
obviously we can not use $u$ as the 4-velocity of the fluid in the
metric $\tilde{g}$. Instead of $u$ we now take a rescalled vector
field 
$$\tilde{u}=(1-sR^2)u,$$ which at each point is in the direction of
$u$ and has a unit norm, $\tilde{g}(\tilde{u},\tilde{u})=-1$, in the metric $\tilde{g}$. Surprisingly $\tilde{g}$ with
such $\tilde{u}$ satisfies the Einstein equations with energy
momentum tensor being a sum of the energy momentum tensor of a 
dust moving along $\tilde{u}$ and the energy momentum of the
  cosmological constant type $\tde_{\mu\nu}=-\tfrac{1}{8\pi
  G}\tilde{\Lambda}\tilde{g}_{\mu\nu}$. More precisely we have the
  following theorem.
\begin{theorem}
Consider Robertson-Walker metrics $g$ as in (\ref{r1}) and
$\tilde{g}$ as in (\ref{r2}). If $g$ satisfies the Friedmann
equations (\ref{fr1})-(\ref{fr2}) for the pure dust moving with the 4-velocity
$u$ in $M$, and having the energy density in the comoving frame
equal to $\rho$, then the metric $\tilde{g}$,
which in $M$ has the same unparametrized geodesics as $g$,
satisfies the 
Einstein equations 
\be
\tilde{E}_{\mu\nu}+\tilde{\Lambda}\tilde{g}_{\mu\nu}=8\pi
G\tilde{T}_{\mu\nu}\label{eei}\ee
for a pure dust,
$\tilde{T}_{\mu\nu}=\tilde{\rho}\tilde{u}_\mu\tilde{u}_\nu$, 
with 4-velocity $\tilde{u}=(1-sR^2)u$, the energy density 
$$\tilde{\rho}=\rho+\frac{s}{8\pi G}\big(\frac{2GM}{R}-4\kappa\big),$$
and the cosmological `constant'
\be
\tilde{\Lambda}=s\big(\kappa-\frac{2GM}{R}\big).\label{cc}\ee
\end{theorem}       
\begin{proof}
We use Proposition \ref{pro}. According to it the nonvanishing
components of the Einstein equations (\ref{eei}) are the diagonal
ones: $\{00\}$ and $\{ij\}$. The $\{00\}$ component gives:
$$E_{00}-3s\kappa-\tilde{\Lambda}=8\pi G\tilde{\rho},$$
and the $\{ij\}$ components give:
$$E_{ij}+s(\kappa+2\dot{R}^2+2R\ddot{R})\delta_{ij}+\tilde{\Lambda}\delta_{ij}=0.$$  
Inserting in these equations the values of $E_{00}$ and $E_{ij}$ from
(\ref{fr1}) we get:
\be
\begin{aligned}
&8\pi G\rho-3s\kappa-\tilde{\Lambda}=8\pi G\tilde{\rho}\\
&s(\kappa+2\dot{R}^2+2R\ddot{R})+\tilde{\Lambda}=0.\end{aligned}\label{uu}
\ee
Now we insert the value 
of $\ddot{R}$ from the second equation (\ref{fr1}) and the value of 
$\dot{R}$ from the Friedman equation (\ref{fr2}) in the second
equation (\ref{uu}). After a simple algebra this proves the formula
(\ref{cc}) for $\tilde{\Lambda}$. Inserting this in the first of
equations (\ref{uu}) proves the formula for $\tilde{\rho}$. This finishes the proof.
\end{proof}
Using this theorem we address the following issue: 
\begin{remark}
Since the measurments in cosmology are based on observations of
photons, other elementary particles, or massive bodies, and since all
of them move along geodesics, it is not clear why, based only on 
observations of geodesics, astronomers, decide to use the
Robertson-Walker metric $g$ to interpret their data. According to 
our analysis they can equally use any metric $\tilde{g}$ with any
value of the parameter $s$, because in all of these metrics the geodesics
look the same: whatever choice of $s$ in $\tilde{g}$ we make 
the Universe is always identified with the 
same manifold $M$, and the geodesics, i.e. the trajectories of all
particles and massive bodies, are the same for all of these choices. 
But if we accept that we can use the metrics
$\tilde{g}$ with $s=0$ and $s\neq 0$ on equal footing, we encounter
the problem what is really the energy content of the Universe. In
particular the celebrated notion of the dark energy becomes
meaningless in such case: the dark energy content is absent in the metric
$\tilde{g}$ with $s=0$ and present in the metric $\tilde{g}$ with
$s\neq 0$.   
\end{remark}

\noindent
{\bf Acknowledgements} I wish to thank Vladimir Matveev for inspiration.

\end{document}